\documentclass[11pt,final]{iopart}
\usepackage{amssymb}
\usepackage{amsfonts}
\usepackage{amsthm}
\usepackage{iopams}
\usepackage[dvips]{epsfig}
\usepackage{float}
\usepackage{footmisc}
\usepackage[title,titletoc,toc]{appendix}
\usepackage{prettyref}
\usepackage{appendix}
\usepackage{showkeys}
\usepackage{todo}
\usepackage{prettyref}
\usepackage{verbatim}
\usepackage[latin1]{inputenc}
\usepackage{subfigure}
\usepackage{verbatim}
\newrefformat{eq}{(\ref{#1})}
\newrefformat{fig}{Figure~\ref{#1}}
\usepackage[pdfpagelabels]{hyperref}
\hypersetup{colorlinks=true,
linkcolor=blue,
citecolor=blue,
plainpages=false}
\newtheorem{theorem}{\bf Theorem}[section]

\newtheorem{assump}[theorem]{\bf Assumption}

\newtheorem{lemma}[theorem]{\bf Lemma}
\newtheorem{eg}[theorem]{\sc Example}
\newtheorem{definition}[theorem]{\bf Definition}
\newtheorem{remark}[theorem]{\bf Remark}

\newtheorem{cor}[theorem]{\bf Corollary}

\newcommand{\argmin}{{\rm arg\,min}}

\eqnobysec

\makeindex

\newcommand{\cB}{{\mathcal B}}

\newcommand{\cD}{{\mathcal D}}

\newcommand{\cH}{{\mathcal H}}

\newcommand{\cL}{{\mathcal L}}
\newcommand{\cM}{{\mathcal M}}

\newcommand{\cO}{{\mathcal O}}

\newcommand{\cT}{{\mathcal T}}

\newcommand{\cV}{{\mathcal V}}
\newcommand{\cW}{{\mathcal W}}

\newcommand{\cZ}{{\mathcal Z}}

\newcommand{\N}{\mathbb{N}}
\newcommand{\R}{\mathbb{R}}

\newcommand{\norm}[1]{|| #1||}

\newcommand{\ba}{\begin{array}}
\newcommand{\ea}{\end{array}}
\newcommand{\be}{\begin{equation}}
\newcommand{\ee}{\end{equation}}
\newcommand{\bea}{\begin{eqnarray}}
\newcommand{\eea}{\end{eqnarray}}
\newcommand{\beq}{\begin{equation}}
\newcommand{\eeq}{\end{equation}}
\newcommand{\bqt}{\begin{quote}}
\newcommand{\eqt}{\end{quote}}

\begin{document}

%-------------------------------------------------------------------------------
%
%-------------------------------------------------------------------------------
\title[Variational Source Condition with the Bregman Distance]
{Generalized Variational Source Condition Associated with the Bregman Distance-I: 
Verification of the Variational Source Condition and Stability of the Total Error Estimation}

%-------------------------------------------------------------------------------
%
%-------------------------------------------------------------------------------
\author{Erdem Altuntac}

\address{Institute for Numerical and Applied Mathematics,
University of G\"{o}ttingen, Lotzestr. 16-18,
D-37083, G\"{o}ttingen, Germany}

%\address{\hspace*{1cm}\\Draft from \today}

\ead{\mailto{e.altuntac@math.uni-goettingen.de}}

%-------------------------------------------------------------------------------
% ABSTRACT
%-------------------------------------------------------------------------------
\begin{abstract}

A general deterministic analysis to state the necessary conditions with
a coefficient determination for the variational source 
condition to hold is provided. Of particular interest in terms
of the choice of the regularization parameter, Morozov's
discrepancy principle enables one to determine
new stable lower and upper bounds for the regularization parameter.
With these bounds, it is also possible to establish quantitative 
estimations for the index function
as well as for the different definitions of the Bregman distance.
Inclusion of the variational source 
condition into the stability analysis enables one 
to re-establish convergence and convergence rate 
results in terms of the index function. The coefficient in the 
variational source condition is explicitly defined as a multivariable 
function of constants in Morozov's discrepancy principle.
As expected, the results here are applicable
when any strictly convex, smooth/non-smooth objective functional is considered.

\end{abstract}

\bigskip

%===========================================================
% Intro
%===========================================================

\section{Introduction}

Variational regularization has commenced by
introducing a new image denoising method named as 
{\em total variation}, \textbf{\cite{RudinOsherFatemi92}}.
Application and analysis of the method have been widely carried
out in the communities of inverse problem and optimization, 
\textbf{\cite{AcarVogel94, BachmayrBurger09, BardsleyLuttman09, ChambolleLions97,
ChanChen06, ChanGolubMulet99, DobsonScherzer96, 
DobsonVogel97, Vogel02, VogelOman96}}. In variational regularization the usage 
of Bregman distance as a tool for the convergence and convergence rate has been
well established over the last decade, 
\textbf{\cite{BurgerOsher04, Grasmair10, Grasmair13, GrasmairHaltmeierScherzer11, 
HofmannMathe12, HofmannYamamoto10, HohageWeidling15, HohageWeidling16, Lorenz08}}. 
As alternative to well known regularization theory 
for minimizing the quadratic Tikhonov functional,
\textbf{\cite{Tikhonov63, TikhonovArsenin77}},
studying convex variational regularization with some general penalty term $J$
has recently become important. 
In a recent work by Hohage {\em et al.} 2015, \textbf{\cite[Eq. (2)]{HohageWeidling15}} 
and references therein, a conventional variational source condition (VSC) with 
a logarithmic index function $\Psi$ has been derived for an inverse scattering problem.
This work is followed up by another research wherein the coefficient verification
of the VSC has been carried out in \textbf{\cite{HohageWeidling16}}.
Authors in \textbf{\cite{HohageWeidling16}} have verified the existence of some
necessary coefficient in the VSC for quadratic Tikhonov functionals under 
some conditions. In this work, we study general type  
Tikhonov functional. We explore under which conditios the VSC hold
and the tight convergence rate results explicitly. The mathematical
development of this work entails the specific rule for
the choice of the regularization paramater
which is Morozov's discrepancy principle.

%Coefficients appear in the VSC and in the 
%regularization parameter have the certain impact on the stabilization
%of the Bregman distance. Thus, some quantitave estimations for
%the total error value will be redefined in terms the index function
%and its coefficient.

Hofmann and Math\'{e} {\em et al.} 2012,
\textbf{\cite{HofmannMathe12}}, {\em a priori} and {\em a posteriori}
strategies for the choice of the regularization parameter
in Banach spaces under the variational source condition 
to determine the total error estimation 
\bea
\label{total_err_est}
E(\varphi_{\alpha(\delta , f^{\delta})}^{\delta} , \varphi^{\dagger}) := 
\norm{\varphi_{\alpha(\delta , f^{\delta})}^{\delta} - \varphi^{\dagger}}_{\cV},
\eea
have been studied extensively. This work does
not necessarily convey any specific solution space $\cV$ since 
the penalty term of our objective functional is not specified.
By establishing some quantitative analysis for the Bregman
distance $D_{J},$ the total error estimation will also be stabilized
owing to the consideration of this work below,
\bea
\label{total_error_bregman}
E(\varphi_{\alpha(\delta , f^{\delta})}^{\delta} , \varphi^{\dagger}) \leq 
D_{J}(\varphi_{\alpha(\delta , f^{\delta})}^{\delta} , \varphi^{\dagger}).
\eea
Therefore, the objective of this work
is to investigate the stable bounds for $D_{J}$ in terms of 
an increasing and positive definite index function $\Psi$ 
depending on the noise amount $\delta$ such that
\bea
\label{regularization_strategy}
\alpha(\delta , f^{\delta}) \rightarrow 0 \mbox{ and } \frac{\delta^2}{\alpha(\delta , f^{\delta})} \rightarrow 0
\mbox{, as } \delta \rightarrow 0.
\eea

Organization of this work is as follows; In the following 
section, we give the necessary preliminaries that are
the base of entire mathematical development. In Section 
\ref{section_norm_convergence}, we review the fundamentals
of Morozov's discrepancy principle. We then move on to the
study of verification of the generalized variational
source condition in conjunction with coefficient determination. 
Those conditions will be related
to the rule for the choice of the regularization parameter.
In Section \ref{interval_for_regpar}, we, in light of our 
{\em a posteriori} choice of the regularization parameter, introduce a
new interval for the value of the regularization
paremeter, {\em i.e.} stable upper and lower bounds.
It is with these bounds that we will be able to stabilize
the total error estimation. Final scientific development
will be given in Section \ref{convergence_by_VSC}.
We also propose a different form the VSC in the Appendix
\ref{symmetric_VSC} by comparing the reverse form of the usual
Bregman distance to the positive definite index function.

%Problem of finding the optimum minimizer for a general 
%Tikhonov type functional is formulated below
%\beq
%\label{problem0}
%\varphi_{\alpha}^{\delta} \in 
%\argmin_{\varphi \in \cV} \left\{\frac{1}{2} \norm{\cT\varphi - f^{\delta}}_{\cH}^2 + \alpha J(\varphi) \right\} .
%\eeq

%===========================================================
% Notations
%===========================================================
\section{Notations and Prerequisite Knowledge}
\label{notations}

%--------------------
% Regularization
%--------------------

\subsection{Assumptions about the forward operator and the penalty term}
\label{notations_regularization_theory}

Denote by $\cV$ and $\cH$ some reflexive/non-reflexive Banach 
and Hilbert spaces respectively. 
For the given linear, injective and compact
forward operator $\cT : \cD(\cT) = \cV \rightarrow \cH ,$
we consider solving a linear ill-posed operator equation
formulated by
\bea
\label{operator_problem}
\cT\varphi = f^{\dagger}.
\eea
The usual inverse problem is that of reconstruction
of the approximate solution $\varphi_{\alpha}^{\delta}$ 
by minimizing a general Tikhonov type objective functional
\bea
\label{cost_functional}
F_{\alpha} : & \cV \times \cH & \longrightarrow \R_{+} ,
\nonumber\\
& (\varphi , f^{\delta}) & \longmapsto 
F_{\alpha}(\varphi , f^{\delta}) := 
\frac{1}{2}\norm{\cT\varphi - f^{\delta}}_{\cH}^2 + \alpha J(\varphi) ,
\eea
from the given data $f^{\delta}$ of the exact right-hand side $f^{\dagger} \in \cH$ with
\begin{displaymath}
f^{\delta} \in \cB_{\delta}(f^{\dagger}), \textit{\mbox{ i.e. }} 
\norm{f^{\dagger} - f^{\delta}}_{\cH} \leq \delta .
\end{displaymath}
In (\ref{cost_functional}), the nonsmooth $J : \cV \rightarrow \R_{+}$
is the convex regularizer with the regularization 
parameter $\alpha > 0$ before it.
It is assumed that the any non-zero constant function
under the image of the forward operator does not vanish, 
and this fact can be formulated as follows,
\bea
\label{forward_op_constant}
\cT 1 \neq 0.
\eea
The real valued solution function $\varphi$ is defined on a compact
domain $\Omega.$  

%----------------------------------
% The existence and the uniqueness
%----------------------------------

\subsection{The existence and the uniqueness of the minimizer}
\label{existence_uniqueness}

Our argument on the existence and the uniqueness of the minimizer
is rather pre-assumptional since this work aims to provide some 
general analysis. Throughout the available literature, {\em e.g.} 
\textbf{\cite[p. 2-3]{HofmannMathe12}}, 
\textbf{\cite[3rd of Assumption 2.1]{HofmannYamamoto10}}, 
\textbf{\cite[4th of Assumption 3.13]{ScherzerGrasmair09}},
the sublevel sets of the objective functional $F_{\alpha},$
or of the penalty term $J,$ have been assumed to be 
sequentially pre-compact. However, if one considers
the penalty term as
\bea
J_{\beta}^{\mathrm{TV}}(\varphi) := \int_{\Omega} \sqrt{\vert \nabla\varphi(x) \vert_2^2 + \beta} dx,
\nonumber
\eea
then it can be shown in a counterexample that the sublevel sets for 
$J_{\beta}^{TV}$ are not sequentially pre-compact.

\begin{eg} 
According to \textbf{\cite[p. 2]{HofmannMathe12}}, the sublevel 
sets for $J_{\beta}^{\mathrm{TV}}$ are defined below,

\begin{displaymath}
\cM_{R}^{J_{\beta}^{\mathrm{TV}}} := \{ \varphi \in \cW^{1,2}(\Omega) : 
J_{\beta}^{\mathrm{TV}}(\varphi) \leq R \}, \mbox{ for } R > 0 .
\end{displaymath}
Obviously $\cM_{R}^{J_{\beta}^{\mathrm{TV}}} \subset BV(\Omega).$
To ensure that the sublevel sets are weakly sequentially compact, 
one must show that every sequence $\varphi_{n} \in \cM_{R}^{J_{\beta}^{\mathrm{TV}}}$ 
has a weakly convergent subsequence with the limit in $\cM_{R}^{J_{\beta}^{\mathrm{TV}}},$
{\em i.e.} the sequence $\varphi_{n} \in \cM^{J_{\beta}^{\mathrm{TV}}}$ has a subsequence 
$\{\varphi_{n_k}\}_{k = 1}^{\infty} \subset \{ \varphi_{n} \}_{n = 1}^{\infty}$ 
such that $\varphi_{n_k} \rightharpoonup \varphi^{\ast}$ as $k \rightarrow \infty$
where $\varphi^{\ast} \in \cM_{R}^{J_{\beta}^{\mathrm{TV}}}.$ However,
it can be shown that the sublevel sets 
$\cM_{R}^{J_{\beta}^{\mathrm{TV}}}$ contain sequence which does
not have weakly convergent subsequence. To do so, for some real function
$\tilde{\varphi} \in \cM_{R}^{J_{\beta}^{\mathrm{TV}}},$ consider the
sequence $\varphi_n(x) = \tilde{\varphi}(x) + n 1,$ where $x \in \Omega.$
Although, for the defined sequence 
$J_{\beta}^{\mathrm{TV}}(\varphi_n) = J_{\beta}^{\mathrm{TV}}(\tilde{\varphi}) \leq R$
holds for any $n \in \N$ the sequence $\varphi_n$ cannot have weakly
convergent subsequence in $BV(\Omega)$ since

\bea
\norm{\varphi_n - \varphi_m}_{BV(\Omega)} & \geq & \norm{\varphi_n - \varphi_m}_{\cL^1(\Omega)}
\nonumber\\
& \geq & \vert n- m \vert \vert \Omega \vert \rightarrow \infty \mbox{ as } n \rightarrow \infty .
\eea

\end{eg}
According to \textbf{\cite[Theorem 3.1]{AcarVogel94}},
in order to ensure the existence of the regularized solution,
one must be able to ensure the $BV$-coercivity of the objective functional 
$F_{\alpha}.$ Usually general type of Tikhonov functionals are not strictly convex since
the forward operator $\cT$ may not necessarily be injective. In our case,
uniqueness of the regularized solution is a result of the strict convexity of the
objective functional since the forward operator $\cT$ is assumed to be
injective. 

%-------------------------------
% Bregman divergence
%-------------------------------

\subsection{Bregman distance}
\label{bregman_divergence_def}

Conventional thorough feedback for the following terminology can be found
in \textbf{\cite{Burger15, Grasmair13, ScherzerGrasmair09}}.

\begin{definition}\textbf{[Subdifferential]}
Let $J : \cV \rightarrow \R_{+} \cup \{ \infty \}$ be defined on an appropriate Banach space and  
be some convex functional. Then subdifferential $\partial J(u) \subset \cV^{\ast}$
of $J$ at $u \in \cV$ is defined as the set of all $p \in \cV^{\ast}$
satisfying the inequality
\bea
\label{subdiff_ineq}
J(v) - J(u) - \langle p , v - u \rangle \geq 0 \mbox{, for all } v \in \cV .
\eea
Here the element $p \in \partial J(u)$ is called the subgradient.
\end{definition}
Note that when $\partial J(u)$ is a singleton, G\^{a}teaux differentiability
and subdifferentiability of $J$ are equal to each other.

\begin{definition}\textbf{[The Generalized Bregman Distances]}
Let $J : \cV \rightarrow \R_{+} \cup \{ \infty \}$ be a convex functional 
with the subgradient $p \in \partial J(u^{\ast}).$ 
Then, for $u, u^{\ast} \in \cV,$ Bregman distance associated with the functional 
$J$ is defined by
\bea
\label{bregman_divergence_intro}
D_{J} : & \cV \times \cV & \longrightarrow \R_{+}
\nonumber\\
& (u , u^{\ast}) & \longmapsto D_{J}(u , u^{\ast}) := J(u) - J(u^{\ast}) - \langle p , u - u^{\ast} \rangle .
\eea
In addition to the traditional definition of Bregman distance in (\ref{bregman_divergence_intro}),
the {\em symmetric Bregman distance} is also given below, (cf. \textbf{\cite[Definition 2.1]{Grasmair13}}),
\beq
\label{symmetrical_bregman}
D_{J}^{\mathrm{sym}}(u , u^{\ast}) := D_{J}(u , u^{\ast}) + D_{J}(u^{\ast} , u) .
\eeq
From here, one can easily observe that
\bea
\label{symmetric_bregman_ordering}
D_{J}^{\mathrm{sym}}(u , u^{\ast}) \geq D_{J}(u , u^{\ast}) .
\eea
Same also holds if one replaces the right hand side by 
the reverse Bregman distance $D_{J}(u^{\ast} , u).$
\end{definition}

%----------------------------------------------------------------------------
% Regularization parameter, first order optimality condition
%----------------------------------------------------------------------------

\subsection{Minimization problem}

The regularized solution $\varphi_{\alpha}^{\delta}$ is constructed by 
employing an appropriate regularization strategy
for the following convex variational 
minimization problem,
\beq
\label{problem0}
\varphi_{\alpha}^{\delta} \in 
\argmin_{\varphi \in \cV} F_{\alpha}(\varphi , f^{\delta}) .
\eeq
Inherently, this solution satisfies the following 
first order optimality condition, 
(cf. \textbf{\cite[Eq. (3.4)]{Burger15}}),
\bea
\label{optimality_1}
\frac{1}{\alpha} \cT^{\ast}(f^{\delta} - \cT\varphi_{\alpha}^{\delta}) \in \partial J(\varphi_{\alpha}^{\delta}) .
\eea
%Another dual minimization problem to (\ref{problem0}) can also be given by 
%\beq
%\label{constrained_problem}
%J(\varphi) \rightarrow \min_{\varphi \in \cV}
%\mbox{, subject to } \norm{\cT\varphi - f^{\delta}}_{\cH} \leq \delta ,
%\eeq
%see \textbf{\cite[Subsection 3.1]{BurgerOsher04}} for the details.

%Following from the problem \ref{problem0},
%in what follows, the general Tikhonov type cost functional 
%$F_{\alpha} : \cV \times \cH \rightarrow \R_{+}$ 
%with convex penalty term $J: \cV \rightarrow \R_{+}$ is then formulated by
%\bea
%F_{\alpha}(\varphi , f^{\delta}) := 
%\frac{1}{2}\norm{\cT\varphi - f^{\delta}}_{\cH}^2 + \alpha J(\varphi) .
%\nonumber
%\eea

\begin{definition}\textbf{[$J$-minimizing Solution]}
\label{J-min_def}
Let $\cV$ be an appropriate Banach space and $\cH$ be
some Hilbert space. For some given linear, injective and compact 
forward operator $\cT : \cV \rightarrow \cH,$ the
$J$-minimizing solution is a solution to the linear operator equation
\bea
\cT\varphi = f^{\dagger}
\eea
if
\bea
\label{J-min_est}
J(\varphi^{\dagger}) := \min{ \{J(\varphi) : \varphi \in \cV, \mbox{ } \cT\varphi = f^{\dagger} \} }.
\eea
\end{definition}
Although our work rather focuses on determining the stable
upper bounds for the Bregman distance $D_{J},$ it is still
worthwhile to review some norm convergence rates both
in the image and in the pre-image spaces.
Owing to the {\em a posteriori} strategy for the choice
of regularization parameter $\alpha = \alpha(\delta , f^{\delta}),$
see subsection \ref{choice_of_regpar} for the details,
with the deterministic noise model 
$f^{\delta} \in \cB_{\delta}(f^{\dagger})$ 
in the measurement space, the following rates can be quantified;
\begin{enumerate}
\item $\cT\varphi_{\alpha(\delta , f^{\delta})}^{\delta} \in \cB_{\cO(\delta)}(\cT\varphi^{\dagger});$
norm of the discrepancy between 
$\cT\varphi_{\alpha(\delta , f^{\delta})}^{\delta}$ and $\cT\varphi^{\dagger}$
by the rate of $\cO(\delta),$ {\em i.e.} 
$\Vert \cT\varphi_{\alpha(\delta , f^{\delta})}^{\delta} - \cT\varphi^{\dagger} \Vert_{\cH} = \cO(\delta).$

\item $D_{J}(\varphi_{\alpha(\delta , f^{\delta})}^{\delta} , \varphi^{\dagger}) = \cO(\Psi(\delta));$
upper bound for the Bregman distance $D_J.$ 

\item $\varphi_{\alpha(\delta , f^{\delta})}^{\delta} \in \cB_{\cO(\Psi(\delta))}(\varphi^{\dagger});$
convergence of the regularized solution 
$\varphi_{\alpha(\delta , f^{\delta})}^{\delta}$ 
to the true solution $\varphi^{\dagger}$ by the rate of
the noise amount $\cO(\Psi(\delta)),$ {\em i.e.,}
$\norm{\varphi_{\alpha(\delta , f^{\delta})}^{\delta} - \varphi^{\dagger}}_{\cV} = \cO(\Psi(\delta)).$
\end{enumerate}
For derivation of these rates, we refer reader to \textbf{\cite{HofmannMathe12}}
and references therein.

\section{Convex Variational Regularization with the Choice of the Regularization Parameter}
\label{section_norm_convergence} 

It is in this section that we explicitly formulate
the necessary condition for the VSC to hold and
deliver a coefficient determination.

\subsection{Choice of the regularization parameter: Morozov's discrepancy principle}
\label{choice_of_regpar}

We are concerned with asymptotic properties of the
regularization parameter $\alpha$ for the Tikhonov-regularized solution
obtained by Morozov's discrepancy principle (MDP).
MDP serves as an \textit{a posteriori}
parameter choice rule for the Tikhonov type objective functionals  
(\ref{cost_functional}) and has certain impact 
on stabilizing the total error functional 
$E : \cV \times \cV \rightarrow \R_{+}$
having the assumed relation
\bea
E(\varphi_{\alpha(\delta , f^{\delta})}^{\delta} , \varphi^{\dagger})
\leq D_{J}(\varphi_{\alpha(\delta , f^{\delta})}^{\delta} , \varphi^{\dagger}).
\nonumber
\eea
As has been introduced in 
\textbf{\cite[Theorem 3.10]{AnzengruberRamlau10}} 
and \textbf{\cite{AnzengruberRamlau11}},
we use the following set notations
in the theorem formulations that are necessary
to establish the error estimation between the regularized solution
$\varphi_{\alpha(\delta , f^{\delta})}^{\delta}$ and 
the $J$-minimizing solution $\varphi^{\dagger}$ 
respectively for the operator equation (\ref{operator_problem}) 
and for the minimization problem (\ref{problem0}),
\bea
\label{MDP1}
\overline{S} & := & \left\{ \alpha : \norm{\cT\varphi_{\alpha(\delta , f^{\delta})}^{\delta} - f^{\delta}}_{\cH} 
\leq \overline{\tau}\delta \mbox{ for some } 
\varphi_{\alpha}^{\delta} \in \argmin_{\varphi \in \cV} \{ F_{\alpha}(\varphi , f^{\delta})\} \right\} ,
\\
\label{MDP2}
\underline{S} & := & \left\{ \alpha : \underline{\tau}\delta \leq 
\norm{\cT\varphi_{\alpha(\delta , f^{\delta})}^{\delta} - f^{\delta}}_{\cH} \mbox{ for some } 
\varphi_{\alpha}^{\delta} \in \argmin_{\varphi \in \cV} \{ F_{\alpha}(\varphi , f^{\delta})\} \right\} ,
\eea
where the discrepancy set radii 
$1 < \underline{\tau} \leq \overline{\tau} < \infty$ are fixed.
Analogously, also as well known from
\textbf{\cite[Eq. (4.57) and (4.58)]{Engl96}} and
\textbf{\cite[Definition 2.3]{Kirsch11}},
we are interested in such a regularization parameter $\alpha(\delta , f^{\delta}),$ 
with some fixed discrepancy set radii 
$1 < \underline{\tau} \leq \overline{\tau} < \infty ,$ that
\beq
\label{discrepancy_pr_definition}
\alpha(\delta , f^{\delta}) \in \{ \alpha > 0 \mbox{ }\vert \mbox{ }
\underline{\tau}\delta \leq \norm{\cT\varphi_{\alpha(\delta , f^{\delta})}^{\delta} - f^{\delta}}_{\cH} \leq \overline{\tau}\delta \} 
= \overline{S} \cap \underline{S} \mbox{ for the given } (\delta , f^{\delta}).
\eeq
It is also the immediate consequences of MDP that 
the following estimations
\bea
\label{consequence_MDP}
\norm{\cT\varphi_{\alpha(\delta , f^{\delta})}^{\delta} - \cT\varphi^{\dagger}}_{\cH} \leq (\overline{\tau} + 1)\delta,
\\
\label{consequence_MDP2}
(\underline{\tau} - 1)\delta \leq \norm{\cT\varphi_{\alpha(\delta , f^{\delta})}^{\delta} - \cT\varphi^{\dagger}}_{\cH},
\eea
hold true. Furthermore, according to \textbf{\cite[Corollary 2]{HofmannMathe12}},
the regularization parameter 
$\alpha(\delta , f^{\delta}) \in \underline{S}$ can be bounded
below by,
\bea
\label{lower_bound_regpar}
\alpha(\delta , f^{\delta}) \geq 
\frac{1}{4}\frac{\underline{\tau}^2 - 1}{\underline{\tau}^2 + 1}
\frac{\delta^2}{\Psi((\underline{\tau} - 1)\delta)},
\eea
where $\Psi$ is a concave, positive definite index function. 
A new lower bound depending on this index function 
for the regularization parameter will be developed.
With a stable lower bound for $\alpha(\delta , f^{\delta}),$ 
possible singularity is avoided as $\alpha \rightarrow 0,$ {\em e.g.} 
see Lemma \ref{lemma_J_diff1} and Lemma \ref{lemma_J_diff2}.

%-----------------------------------------------------------------
% Variational Inequalities for Norm Convergence
%-----------------------------------------------------------------

\subsection{Generalized variational source condition verification}
\label{variational_ineq}

Convergence rates results for some general operator $\cT$
can be obtained by formulating {\em variational inequality}
which uses the concept of index functions. A function
$\Psi : [0 , \infty) \rightarrow [0 , \infty)$
is called {\em index function} if it is continuous, 
monotonically increasing and $\Psi(0) = 0.$
VSC plays an important role in the 
development of convergence and convergence rate
results for convex variational regularization strategies. 
Verification of this source condition has recently become popular, 
see \textbf{\cite{HohageWeidling15, HohageWeidling16}}.
We rather associate the conventional VSC 
with the generalized Bregman distance since
the objective functional (\ref{cost_functional}) can 
involve any non-smooth and convex functional $J.$ 

\begin{assump}
\label{assump_conventional_variational_ineq}
\textbf{[Variational Source Condition]}
There exists some constant $\sigma \in (0 , 1]$
and a concave index function $\Psi$ such that
\bea
\label{variational_ineq}
\frac{\sigma}{2} D_{J}(\varphi , \varphi^{\dagger}) \leq  J(\varphi) - J(\varphi^{\dagger})
+ \Psi\left( \norm{\cT\varphi - \cT\varphi^{\dagger}}_{\cH} \right)  \mbox{, for all }
\varphi \in \cV .
\eea 
\end{assump}

Below the necessary condition for the VSC to hold 
and a coefficient determination will be
formulated. The result is applicable for any convex
and smooth/non-smooth penalty term $J$ only
in conjunction with MDP.

\begin{theorem}\label{thrm_VSC_verification}
Consider the choice of the regularization parameter {\em a posteriori}  
$\alpha(\delta , f^{\delta}) \in \overline{S} \cap \underline{S},$
with the given data $f^{\delta} \in \cB_{\delta}(f^{\dagger}),$
for the regularized solution 
$\varphi_{\alpha(\delta , f^{\delta})}^{\delta} \in \cD(F_{\alpha})$
to the problem (\ref{problem0}).
If, for the positive definite, monotonically increasing and concave
index function $\Psi : [0 , \infty) \rightarrow [0 , \infty),$ 
the following condition holds true
\bea
\label{VSC_verification_cond}
\langle p , \varphi^{\dagger} - \varphi_{\alpha(\delta , f^{\delta})}^{\delta} \rangle 
\leq C(\underline{\tau} , \overline{\tau}) \Psi(\delta),
\eea
where $p \in \partial J(\varphi^{\dagger})$ and 
$C(\underline{\tau} , \overline{\tau}) : (1 , \infty) \times [\underline{\tau} , \infty) \rightarrow \R_{+} $
then the $J$-minimizing solution $\varphi^{\dagger},$ for some 
$\tilde{\sigma}(\underline{\tau} , \overline{\tau}) : (1 , \infty) \times [\underline{\tau} , \infty) \rightarrow (0 , 1) ,$
satisfies the VSC as below, 
\bea
\label{_thrmvariational_ineq}
\tilde{\sigma} D_{J}(\varphi_{\alpha(\delta , f^{\delta})}^{\delta} , \varphi^{\dagger}) \leq  
J(\varphi_{\alpha(\delta , f^{\delta})}^{\delta}) - J(\varphi^{\dagger})
+ \Psi\left( \norm{\cT\varphi_{\alpha(\delta , f^{\delta})}^{\delta} - \cT\varphi^{\dagger}}_{\cL^{2}(\cZ)} \right).
\eea 
\end{theorem}

\begin{proof}
Firstly, observe that for the fixed $1 < \underline{\tau} \leq \overline{\tau} < \infty$
discrepancy radii,
\bea
1 < 1 + \frac{1}{\underline{\tau} - 1} = \frac{\underline{\tau}}{\underline{\tau} - 1} \leq 
\frac{\overline{\tau}}{\underline{\tau} - 1},
\nonumber
\eea
which implies
\bea
\frac{\underline{\tau} - 1}{\overline{\tau}} < 1.
\nonumber
\eea
This will be beneficial to the coefficient estimation.
Now, for the monotonically increasing and concave index function 
$\Psi : [0 , \infty) \rightarrow [0 , \infty),$
we can estimate
\bea
\Psi(\delta) \leq ^{\footnotemark} \Psi\left( \frac{1}{\underline{\tau} - 1} 
\norm{\cT\varphi_{\alpha(\delta , f^{\delta})}^{\delta} - \cT\varphi^{\dagger}}_{\cH} \right)
& \leq & ^{\footnotemark} \Psi\left( \frac{\overline{\tau}}{\underline{\tau} - 1} 
\norm{\cT\varphi_{\alpha(\delta , f^{\delta})}^{\delta} - \cT\varphi^{\dagger}}_{\cH} \right)
\nonumber\\
& \leq & ^{\footnotemark} \frac{\overline{\tau}}{\underline{\tau} - 1} 
\Psi\left( \norm{\cT\varphi_{\alpha(\delta , f^{\delta})}^{\delta} - \cT\varphi^{\dagger}}_{\cH} \right),
\nonumber
\eea
\footnotetext[1]{by (\ref{consequence_MDP2})}
\footnotetext[2]{$\Psi$ is monotone increasing}
\footnotetext[3]{by the concavity of $\Psi$}
\noindent holds true. On the other hand, convexity of the penalty term $J$ implies
\bea
J(\varphi^{\dagger}) - J(\varphi_{\alpha(\delta , f^{\delta})}^{\delta}) \leq 
\langle p , \varphi^{\dagger} - \varphi_{\alpha(\delta , f^{\delta})}^{\delta} \rangle
\leq \frac{\overline{\tau}}{\underline{\tau} - 1} 
\Psi\left( \norm{\cT\varphi_{\alpha(\delta , f^{\delta})}^{\delta} - \cT\varphi^{\dagger}}_{\cH} \right),
\eea
where $p \in \partial J(\varphi^{\dagger})$ and 
$C(\underline{\tau} , \overline{\tau}) := \frac{\overline{\tau}}{\underline{\tau} - 1}.$
Adding the $J$-difference 
$J(\varphi_{\alpha(\delta , f^{\delta})}^{\delta}) - J(\varphi^{\dagger})$
and taking into consideration the Bregman distance definition (\ref{bregman_divergence_intro})
provides
\bea
D_{J}(\varphi_{\alpha(\delta , f^{\delta})}^{\delta} , \varphi^{\dagger})
\leq J(\varphi_{\alpha(\delta , f^{\delta})}^{\delta}) - J(\varphi^{\dagger})
+ \frac{\overline{\tau}}{\underline{\tau} - 1} 
\Psi\left( \norm{\cT\varphi_{\alpha(\delta , f^{\delta})}^{\delta} - \cT\varphi^{\dagger}}_{\cH} \right),
\eea
where $1 < \underline{\tau} \leq \overline{\tau} < \infty$ are fixed.
By defining the coefficient 
\bea
\label{VSC_coeff}
\tilde{\sigma} := \frac{\underline{\tau} - 1}{\overline{\tau}} < 1,
\eea
the VSC has been verified.
\end{proof}

\begin{remark}
Note that the coefficient defined by (\ref{VSC_coeff}) does not violate
the conventional coefficient condition in the VSC presented in 
(\ref{variational_ineq}), {\em i.e.} the VSC holds for $\sigma \in (0 , 1].$ 
By defining the coefficient $\tilde{\sigma}$ 
as a multivariable function of discrepancy radii, we only intended to
give a general definition for the coefficient. If one sets 
$\overline{\tau} = \underline{\tau} = \tau^{\ast} ,$ then the coefficient 
$\tilde{\sigma} \equiv \tilde{\sigma}(\tau^{\ast}) : (1 , \infty) \rightarrow (0 , 1)$ 
boils down to a single variable function.
\end{remark}

%===========================================
% Bounds for the regularization paratemeter
%===========================================

\section{New Bounds for the Regularization Parameter $\alpha(\delta , f^{\delta})$}
\label{interval_for_regpar}

MDP brings new stable lower and upper bounds for the
regularization parameter.
Let us consider the solution space as a reflexive Banach space 
$\cV = \cL^{2}(\Omega).$ Having motivated by the condition stated in 
(\ref{VSC_verification_cond}), a global estimation for the coefficient $\sigma$ 
as a result of the Bregman distance definition
(\ref{bregman_divergence_intro}) and of the variational source condition
can be derived,
\bea
\label{derivative_to_index}
\frac{\sigma}{2} \langle p , \varphi^{\dagger} - \varphi \rangle \leq \Psi\left( \norm{\cT\varphi - \cT\varphi^{\dagger}}_{\cH} \right)
\mbox{, for } p \in \partial J(\varphi^{\dagger}) \mbox{ and } \varphi \in \cV.
\eea
Furthermore, since the penalty term $J : \cL^{2}(\Omega) \rightarrow \R_{+}$ is convex
then a lower bound, which follows from (\ref{derivative_to_index}), 
for the index function $\Psi$ can be given in terms
of $J$-difference as such
\bea
\label{J-diff_to_index}
\frac{\sigma}{2} \left( J(\varphi^{\dagger}) - J(\varphi) \right) \leq 
\Psi\left( \norm{\cT\varphi - \cT\varphi^{\dagger}}_{\cH} \right)
\eea
On the other hand $\varphi_{\alpha}^{\delta}$ is the minimizer of the
objective functional (\ref{cost_functional}). Thus the estimation (\ref{J-diff_to_index})
reads,
\bea
\frac{\sigma}{4\alpha} \norm{\cT\varphi_{\alpha}^{\delta} - f^{\delta}}_{\cH}^{2}
- \frac{\sigma}{4}\frac{\delta^2}{\alpha} \leq 
\frac{\sigma}{2}\left( J(\varphi^{\dagger}) - J(\varphi_{\alpha}^{\delta}) \right) \leq 
\Psi\left( \norm{\cT\varphi_{\alpha}^{\delta} - \cT\varphi^{\dagger}}_{\cH} \right),
\eea
which implies
\bea
\label{index_func_lower_bound0}
\frac{\sigma}{4\alpha} \norm{\cT\varphi_{\alpha}^{\delta} - f^{\delta}}_{\cH}^{2}
- \frac{\sigma}{4}\frac{\delta^2}{\alpha} \leq 
\Psi\left( \norm{\cT\varphi_{\alpha}^{\delta} - \cT\varphi^{\dagger}}_{\cH} \right).
\eea
A new lower bound for the regularization parameter
will rise from this global estimation. As also well 
known by the literature ({\em e.g} \textbf{\cite{HofmannMathe12}})
lower bound is crucial to control the trade off between 
$\delta$ and $\alpha.$ Unlike in the aforementioned literature,
our lower bound contains the coefficient $\sigma$ from 
(\ref{variational_ineq}) and has a simpler form.

\begin{theorem}\label{new_lower_bound_regpar}
Let the regularization parameter $\alpha = \alpha(\delta , f^{\delta}),$ 
for the minimizer 
$\varphi_{\alpha(\delta , f^{\delta})}^{\delta} \in \cD(F_{\alpha})$ 
of the objective functional $F_{\alpha}$ in (\ref{cost_functional}), 
be chosen according to the discrepancy principle 
$\alpha(\delta , f^{\delta}) \in \overline{S} \cap \underline{S}$ where
the given data $f^{\delta} \in \cB_{\delta}(f^{\dagger}).$ Then this
choice of regularization parameter, for $\sigma \in (0, 1]$ and for the fixed
$1 < \underline{\tau} \leq \overline{\tau} < \infty$ coefficients,
implies the following lower bound for the regularization parameter 
$\alpha(\delta , f^{\delta}),$
\bea
\label{index_func_lower_bound1}
\frac{\sigma}{4}(\underline{\tau} - 1) \frac{\delta^2}{\Psi(\delta)} \leq \alpha(\delta , f^{\delta}).
\eea
\end{theorem}

\begin{proof}
If the regularization parameter is chosen
according to the discrepancy principle 
$\alpha(\delta , f^{\delta}) \in \overline{S} \cap \underline{S}$ where
the given data $f^{\delta} \in \cB_{\delta}(f^{\dagger}),$ then 
it follows from (\ref{index_func_lower_bound0}) that
\bea
\frac{\sigma\underline{\tau}\delta^2}{4\alpha(\delta , f^{\delta})} - \frac{\sigma}{4}\frac{\delta^2}{\alpha(\delta , f^{\delta})} \leq 
\Psi\left( \norm{\cT\varphi_{\alpha(\delta , f^{\delta})}^{\delta} - \cT\varphi^{\dagger}}_{\cH} \right).
\eea
Recall that the index function $\Psi$ is a concave function. We then, from (\ref{consequence_MDP}),
conclude that
\bea
\label{index_func_lower_bound2}
\frac{\sigma}{4}\frac{(\underline{\tau}^2 - 1)}{(\overline{\tau} + 1)}\frac{\delta^2}{\alpha(\delta , f^{\delta})} 
= \frac{\sigma}{4}(\underline{\tau} - 1) \frac{\delta^2}{\alpha(\delta , f^{\delta})} \leq \Psi(\delta)
\mbox{, for } 1 < \underline{\tau} \leq \overline{\tau} < \infty .
\eea
\end{proof}

%\begin{lemma}
%Let the regularization parameter $\alpha(\delta , f^{\delta})$ be chosen 
%{\em a posteriori} for the regularized solution 
%$\varphi_{\alpha(\delta , f^{\delta})}^{\delta} \in \cD(F_{\alpha}).$ Then there exists
%a positive constant $C$ such that
%\bea
%\label{misfit_lower_bound}
%\norm{\cT\varphi_{\alpha(\delta , f^{\delta})}^{\delta} - f^{\delta}}_{\cL^2(\cZ)} \geq C \alpha(\delta , f^{\delta}) ,
%\eea
%holds
%\end{lemma}
%
%\begin{proof}
%Let us assume the opposite, {\em i.e.} there exists some $C$ such that
%\bea
%\norm{\cT\varphi_{\alpha(\delta , f^{\delta})}^{\delta} - f^{\delta}}_{\cL^2(\cZ)} < C \alpha(\delta , f^{\delta}).
%\nonumber
%\eea
%Since the regularization parameter is chosen according to Morozov's discrepancy principle,
%then for the fixed coefficients $1 < \underline{\tau} < \infty,$
%\bea
%
%\eea
%\end{proof}

%The fact 
%$F_{\alpha}(\varphi_{\alpha}^{\delta} , f^{\delta}) \leq F_{\alpha}(\varphi^{\dagger} , f^{\delta})$
%provides the following upper bound in terms of the noise amount $\delta$
%and the single valued regularized $J : \cV \rightarrow \R_{+},$
%\bea
%\label{regpar_upper_bound}
%\alpha \leq \frac{1}{2} \frac{\delta^2}{J(\varphi_{\alpha}^{\delta}) - J(\varphi^{\dagger})} \mbox{, for some } \alpha > 0.
%\eea
%If we plug this upper bound into (\ref{index_func_lower_bound1}), we obtain the desired result.

From the assertion above, a stable lower bound for the index function 
can also be obtained. However, this rises the question of
a stable maximum value of the regularization parameter.
Regardless of the choice of regularization
parameter, there exists some $\delta_{\mathrm{max}} > \delta$
such that $\alpha < \alpha_{\mathrm{max}} = \alpha(\delta_{\mathrm{max}}).$
Analogous to
\textbf{\cite[Eq (3.2) of Proposition 3.1]{HofmannYamamoto10}}, 
we will estimate an improvised form of this maximum value $\alpha_{\mathrm{max}}$
in consideration of the introduced lower bound (\ref{index_func_lower_bound1}).

\begin{theorem}\label{thrm_regpar_max}
Provided that the regularization parameter 
$\alpha(\delta , f^{\delta}) \in \underline{S} \cap \overline{S}$
for the regularized solution $\varphi_{\alpha(\delta , f^{\delta})}^{\delta}$
of the problem (\ref{problem0}) and $\varphi^{\dagger}$ is the 
$J$-minimizing solution (\ref{J-min_def}), then there can be 
defined a maximum value for the regularization parameter 
depending on the positive definite and concave index function
$\Psi : [0 , \infty) \rightarrow [0 , \infty),$
\bea
\label{upper_bound_regpar}
\alpha_{\mathrm{max}} := \left( \cO(\Psi(\delta)) + J(\varphi^{\dagger}) \right)^{-1}
\eea
such that $F_{\alpha_{\mathrm{max}}} < \infty.$
\end{theorem}

\begin{proof}
Let us consider the following 
form of the objective functional
\bea
\label{max_objective_func}
F_{\alpha_{\mathrm{max}}}(\varphi_{\alpha}^{\delta} , f^{\delta}) = 
\frac{1}{2} \norm{\cT\varphi_{\alpha}^{\delta} - f^{\delta}}_{\cH}^2 + 
\alpha_{\mathrm{max}} J(\varphi_{\alpha}^{\delta}) \mbox{, for some } \alpha > 0.
\eea
It follows from here, for some $\alpha > 0,$ that
\bea
F_{\alpha_{\mathrm{max}}}(\varphi_{\alpha}^{\delta} , f^{\delta}) & \leq & 
\frac{1}{2}\norm{\cT\varphi_{\alpha}^{\delta} - f^{\delta}}_{\cH}^2 + \delta^2 +
\alpha_{\mathrm{max}} J(\varphi_{\alpha}^{\delta})
\nonumber\\
& = & F_{\alpha}(\varphi_{\alpha}^{\delta} , f^{\delta}) + 
(\alpha_{\mathrm{max}} - \alpha) J(\varphi_{\alpha}^{\delta}) + \delta^2 
\nonumber\\
& \leq & F_{\alpha}(\varphi^{\dagger} , f^{\delta}) + 
(\alpha_{\mathrm{max}} - \alpha) J(\varphi_{\alpha}^{\delta}) + \delta^2
\nonumber\\
& \leq & \frac{\delta^2}{2} + \alpha J(\varphi^{\dagger}) +
(\alpha_{\mathrm{max}} - \alpha) J(\varphi_{\alpha}^{\delta}) + \delta^2
\nonumber
\eea
Since $\varphi_{\alpha}^{\delta} \in \cD(F_{\alpha}),$
then $\alpha J(\varphi_{\alpha}^{\delta}) \leq \frac{\delta^2}{2} + \alpha J(\varphi^{\dagger})$ 
for some $\alpha > 0.$ Furthermore $\frac{\alpha_{\mathrm{max}}}{\alpha} \geq 1.$
Thus, these facts yield that
\bea
F_{\alpha_{\mathrm{max}}}(\varphi_{\alpha}^{\delta} , f^{\delta}) & \leq &
\frac{\delta^2}{2} + \alpha J(\varphi^{\dagger}) + 
(\alpha_{\mathrm{max}} - \alpha) J(\varphi^{\dagger}) + \frac{3\delta^2}{2}
\nonumber\\
& = & 2\delta^2 + \alpha_{\mathrm{max}} J(\varphi^{\dagger})
\nonumber\\
& \leq & 2\delta^2 \frac{\alpha_{\mathrm{max}}}{\alpha} + \alpha_{\mathrm{max}} J(\varphi^{\dagger})
\nonumber\\
& = & \alpha_{\mathrm{max}} \left( 2\frac{\delta^2}{\alpha} + J(\varphi^{\dagger}) \right) 
\mbox{, for some } \alpha > 0.
\nonumber
\eea
We proceed with this estimation by making use of the 
lower bound estimated in (\ref{index_func_lower_bound1})
as such,
\bea
F_{\alpha_{\mathrm{max}}}(\varphi_{\alpha}^{\delta} , f^{\delta}) 
\leq \alpha_{\mathrm{max}}\left( \frac{8}{\sigma}(\underline{\tau} - 1)\Psi(\delta) 
+ J(\varphi^{\dagger}) \right).
\nonumber
\eea
The result is hence obtained by defining 
\bea
\label{max_val_def_regpar}
\alpha_{\mathrm{max}} := \left( \frac{8}{\sigma}(\underline{\tau} - 1)\Psi(\delta) 
+ J(\varphi^{\dagger}) \right)^{-1}.
\eea
\end{proof}

As has been mentioned above, a stable maximum value
for the regularization parameter $\alpha(\delta , f^{\delta})$
yields a stable lower bound for the index function $\Psi$
owing to (\ref{index_func_lower_bound2}).
We, then, close this section with providing the following corollary.

\begin{cor}\label{cor_index_func_lower_bound}
If one plugs the maximum value for the regularization parameter
explicitly defined by (\ref{max_val_def_regpar}) into
(\ref{index_func_lower_bound1}), then one obtains
\bea
\frac{\sigma}{4}\left( \frac{8}{\sigma}(\underline{\tau} - 1)\Psi(\delta) 
+ J(\varphi^{\dagger}) \right)\delta^2 \leq \Psi(\delta).
\nonumber
\eea
Since $\frac{8}{\sigma}(\underline{\tau} - 1)\Psi(\delta) \geq 0,$ 
then a simpler form of the lower bound can be given by
\bea
\label{index_func_lower_bound2}
\frac{\sigma}{4} J(\varphi^{\dagger}) \delta^2 \leq \Psi(\delta).
\eea
%
%The convex penalty is positive definite, and the discrepancy set 
%radius $\overline{\tau} + 1 > 1.$ This follows from 
%(\ref{index_func_lower_bound2})
%\bea
%\frac{\sigma}{4} \left( J(\varphi^{\dagger}) - J(\varphi_{\alpha(\delta , f^{\delta})}^{\delta}) \right)\delta^2
%\leq \left( \overline{\tau} + 1 \right)\Psi(\delta).
%\eea
%To be in complience with (\ref{index_func_lower_bound0}),

\end{cor}

%===========================================================================
% Convergence
%===========================================================================
\section{Contribution of the VSC to Stabilize the Bregman Distance}
\label{convergence_by_VSC}

%In \textbf{\cite[Eq. (1.4)]{HofmannMathe12}}
%the total error functional has been assigned to the Bregman distance.
%That is an upper bound for the Bregman distance is 
%the quantitative error estimation for the distance between the
%regularized and the $J$-minimizing solutions, $\varphi_{\alpha}^{\delta}$
%and $\varphi^{\dagger}$ respectively. 

%Assumption \ref{assump_conventional_variational_ineq}
%must be regarded as a motivation for the desired error estimation.
%However, 

As has been motivated above in the subsection \ref{choice_of_regpar},
our choice of regularization parameter must fulfill (\ref{regularization_strategy}).
Moving on fom here and together with (\ref{new_lower_bound_regpar}), 
we will obtain stable upper bounds for the Bregman distance 
$D_J$, or for the total error value functional $E,$ see (\ref{total_error_bregman}).
We will also see that it is also possible to bound the reverse
Bregman distance $D_{J}(\varphi^{\dagger} , \varphi).$ 
With this upper bound, we will eventually arrive at the quantitative 
estimation for the symmetric Bregman
distance $D_{J}^{\mathrm{sym}}.$
Therefore, the important question to be answered is how to control
the trade-off between the noise amount $\delta$ and the regularization
parameter $\alpha.$ It will be observed that this controllability
is only possible when the choice of the regularization
parameter is specified which is Morozov's discrepancy principle in our case.
As a result of this choice and of the inclusion of the VSC, 
the quantitave estimations for the Bregmans distance
depend on the discrepancy set radii and the coefficient in the VSC.
In this section, the function space of the measured data
will be taken as $\cL^2(\cZ)$ where $\cZ = \cD(f^{\delta}).$

\begin{lemma}
\label{lemma_J_diff1}
Let the regularization parameter $\alpha = \alpha(\delta , f^{\delta}),$ 
for the minimizer 
$\varphi_{\alpha(\delta , f^{\delta})}^{\delta} \in \cD(F_{\alpha})$ 
of the objective functional $F_{\alpha}$ in (\ref{cost_functional}), 
be chosen according to the disrepancy principle 
$\alpha(\delta , f^{\delta}) \in \overline{S} \cap \underline{S}$ where
the given data $f^{\delta} \in \cB_{\delta}(f^{\dagger}).$
Furthermore, suppose that the $J$-minimizing solution obeys the VSC 
(\ref{variational_ineq}). Then, this {\em a posteriori} rule for the 
choice of the regularization parameter stabilises the following $J$-difference
\bea
\label{J_diff1}
J(\varphi_{\alpha(\delta , f^{\delta})}^{\delta}) - J(\varphi^{\dagger}) = \cO(\Psi(\delta)) .
\eea
\end{lemma}

\begin{proof}
Since $\varphi_{\alpha}^{\delta} \in \cD(F_{\alpha})$
is the minimizer of the objective functional $F_{\alpha},$ 
for some $\alpha > 0$ and for any
$\varphi \in \cD(F_{\alpha})$ it holds that 
$F_{\alpha}(\varphi_{\alpha}^{\delta}) \leq F_{\alpha}(\varphi).$ This 
implies the following,
\bea
\label{J-diff}
J(\varphi_{\alpha}^{\delta}) - J(\varphi^{\dagger}) \leq \frac{\delta^2}{2\alpha} .
\eea
Here, the decrease in $\alpha$ will cause a blow-up on the right hand side. 
This is controlled by the choice of the regularization parameter 
$\alpha = \alpha(\delta , f^{\delta}) \in \underline{S}.$ Thus, we make use 
of the lower bound for the regularization parameter given in (\ref{lower_bound_regpar})
to have a stable upper bound by using the facts that $\Psi$ is a concave
and increasing function,
\bea
\label{trade_off_control}
\frac{\delta^2}{2\alpha(\delta , f^{\delta})} \leq 2 \frac{\underline{\tau}^2 + 1}{\underline{\tau}^2 - 1}\Psi((\underline{\tau} - 1)\delta)
\leq ^{\footnotemark} 2 \frac{\underline{\tau}^2 + 1}{\underline{\tau}^2 - 1} \Psi((\underline{\tau} + 1)\delta)
\leq ^{\footnotemark} 2 \frac{\underline{\tau}^2 + 1}{\underline{\tau} - 1} \Psi(\delta)
\eea
\footnotetext[1]{Since $\Psi$ is an increasing function, then $\Psi((\underline{\tau} - 1)\delta) \leq \Psi((\underline{\tau} + 1)\delta).$}
\footnotetext[2]{Due to the concavity of $\Psi,$ $\Psi((\underline{\tau} + 1)\delta) \leq (\underline{\tau} + 1)\Psi(\delta)$ holds,
see \textbf{\cite[Eq. 2.3 of Proposition 1]{HofmannMathe12}}.}
\noindent Hence, this control over the trade-off between $\delta^2$ and $\alpha$
yields the desired result.
\end{proof}

The lemma above is comparable to its counterparts in the literature,
\textbf{\cite[Corollary 4.2]{AnzengruberRamlau10}, 
\cite[Lemma 2.8]{AnzengruberRamlau11}, 
\cite[Eq. (2.17)]{AnzengruberHofmannMathe14}, 
\cite[Theorem 4.4]{Grasmair13} 
\cite[Lemma 1]{HofmannMathe12}}. We, below, reformulate the
result with a new proof since a new lower bound for the regularization 
parameter that has been stated in Theorem
\ref{new_lower_bound_regpar} will be included.

\begin{lemma}
\label{lemma_J_diff2}
Let the regularization parameter $\alpha = \alpha(\delta , f^{\delta}),$ 
for the minimizer 
$\varphi_{\alpha(\delta , f^{\delta})}^{\delta} \in \cD(F_{\alpha})$ 
of the objective functional $F_{\alpha}$ in (\ref{cost_functional}), 
be chosen according to the disrepancy principle 
$\alpha(\delta , f^{\delta}) \in \overline{S} \cap \underline{S}$ where
the given data $f^{\delta} \in \cB_{\delta}(f^{\dagger}).$
Then this {\em a posteriori} rule for the choice of the regularization
parameter stabilises the following $J$ difference
\bea
\label{J_diff2}
J(\varphi_{\alpha(\delta , f^{\delta})}^{\delta}) - J(\varphi^{\dagger}) = \cO(\Psi(\delta)) .
\eea
\end{lemma}

\begin{proof}
Likewise before, since $F_{\alpha}(\varphi_{\alpha}^{\delta}) \leq F_{\alpha}(\varphi),$
\bea
\label{J-diff3}
J(\varphi_{\alpha}^{\delta}) - J(\varphi^{\dagger}) \leq \frac{\delta^2}{2\alpha} ,
\eea
holds true for some $\alpha > 0.$ The choice of regularization parameter,
as we have seen in Theorem \ref{new_lower_bound_regpar}, provides the stable
lower bound (\ref{index_func_lower_bound1}). Plugging that lower bound
into (\ref{J-diff3}) stabilizes the $J$-difference as such,
\bea
\label{trade_off_control2}
\frac{\delta^2}{2\alpha(\delta , f^{\delta})} \leq \frac{2}{\sigma}\frac{1}{(\underline{\tau} - 1)}\Psi(\delta).
\eea
\end{proof}

Now tight rates for the total error estimation can be established.
We will present two results, one of which is for the usual Bregman distance
and the other one is for its reverse form. These results will inherently
lead to the stable upper bound for the symmetric Bregman distance
that has been defined by (\ref{symmetrical_bregman}).

\begin{theorem}
\label{theorem_Bregman_bound_usual}
Let the $J$-minimizing solution $\varphi^{\dagger} \in \cV$ 
for the operator equation (\ref{operator_problem})
satisfy Assumption \ref{assump_conventional_variational_ineq}.
Under the same conditions in Lemma \ref{lemma_J_diff2}, 
we then have
\bea
\label{theorem_upper_bound_bregman1}
D_{J}(\varphi_{\alpha(\delta , f^{\delta})}^{\delta} , \varphi^{\dagger}) = \cO(\Psi(\delta)) ,
\eea
as $\delta \rightarrow 0.$ 
\end{theorem}

\begin{proof}
Since the true solution $\varphi^{\dagger}$ satisfies
Assumption \ref{assump_conventional_variational_ineq},
\bea
D_{J}(\varphi_{\alpha(\delta , f^{\delta})}^{\delta} , \varphi^{\dagger}) & \leq & 
J(\varphi_{\alpha(\delta , f^{\delta})}^{\delta}) - J(\varphi^{\dagger}) + 
\Psi\left( \norm{\cT\varphi_{\alpha(\delta , f^{\delta})}^{\delta} - \cT\varphi^{\dagger}}_{\cL^{2}(\cZ)} \right).
\nonumber\\
& \leq & \frac{2}{\sigma}\frac{1}{(\underline{\tau} - 1)}\Psi(\delta) + 
\Psi\left( (\overline{\tau} + 1)\delta \right)
\nonumber\\
&\leq & \frac{2}{\sigma}\frac{1}{(\underline{\tau} - 1)}\Psi(\delta) +
(\overline{\tau} + 1)\Psi(\delta).
\eea
The first term on the right hand side, the bound for the $J$ difference, 
comes from Lemma \ref{lemma_J_diff2}. As in the estimation (\ref{trade_off_control2}) 
of Lemma \ref{lemma_J_diff2}, $\Psi$ is concave function, thus 
$\Psi\left( (\overline{\tau} + 1)\delta \right) \leq (\overline{\tau} + 1) \Psi(\delta).$
\end{proof}

\begin{theorem}
\label{theorem_norm_convergence}
Let the regularization parameter $\alpha = \alpha(\delta , f^{\delta}),$ 
for the minimizer 
$\varphi_{\alpha(\delta , f^{\delta})}^{\delta} \in \cD(F_{\alpha})$ 
of the objective functional $F_{\alpha}$ in (\ref{cost_functional}), 
be chosen according to the disrepancy principle 
$\alpha(\delta , f^{\delta}) \in \overline{S} \cap \underline{S}$ where
the given data $f^{\delta} \in \cB_{\delta}(f^{\dagger}).$
Suppose that the $J-$minimizing solution $\varphi^{\dagger} \in \cV,$ 
where $\cT\varphi^{\dagger} = f^{\dagger},$ satisfies 
Assumption \ref{assump_conventional_variational_ineq}
with the concave and monotonically increasing index function 
$\Psi : [0 , \infty) \rightarrow [0 , \infty) .$
Then, this {\em a posteriori} rule for the choice of regularization
parameter yields the following rate,
\begin{eqnarray}
\label{theorem_upper_bound_bregman}
D_{J}(\varphi^{\dagger} , \varphi_{\alpha(\delta , f^{\delta})}^{\delta}) = \cO(\Psi(\delta)) ,
\end{eqnarray}
as $\delta \rightarrow 0.$ 
\end{theorem}

\begin{proof}

Firstly, by Assumption \ref{assump_conventional_variational_ineq},
it can easily be observed that,
\bea
J(\varphi^{\dagger}) - J(\varphi_{\alpha(\delta , f^{\delta})}^{\delta}) \leq 
\Psi\left( \norm{\cT\varphi_{\alpha(\delta , f^{\delta})}^{\delta} - \cT\varphi^{\dagger}}_{\cL^{2}(\cZ)} \right) .
\nonumber
\eea
From the early observation (\ref{consequence_MDP}) and since $\Psi$ is a
monotonically increasing, concave function, we obtain,
\bea
\label{J_diff2}
J(\varphi^{\dagger}) - J(\varphi_{\alpha(\delta , f^{\delta})}^{\delta}) \leq 
\Psi\left( (\overline{\tau} + 1)\delta \right) \leq (\overline{\tau} + 1)\Psi(\delta).
\eea
Regarding the aimed upper bound for the Bregman distance, 
use the estimation (\ref{J_diff2}) for $\alpha(\delta , f^{\delta}) \in \overline{S}$ 
and observe the following for $p \in \partial J(\varphi_{\alpha(\delta , f^{\delta})}^{\delta}),$
\bea
D_{J}(\varphi^{\dagger} , \varphi_{\alpha(\delta , f^{\delta})}^{\delta}) & = & 
J(\varphi^{\dagger}) - J(\varphi_{\alpha(\delta , f^{\delta})}^{\delta}) 
- \langle p , \varphi^{\dagger} - 
\varphi_{\alpha(\delta , f^{\delta})}^{\delta} \rangle
\nonumber\\
& = & J(\varphi^{\dagger}) - J(\varphi_{\alpha(\delta , f^{\delta})}^{\delta})
+ \langle p , \varphi_{\alpha(\delta , f^{\delta})}^{\delta} - \varphi^{\dagger}  \rangle
\nonumber\\
& \leq & (\overline{\tau} + 1)\Psi(\delta)
+ \langle p , \varphi_{\alpha(\delta , f^{\delta})}^{\delta} - \varphi^{\dagger}  \rangle
\eea
Since the regularized solution
$\varphi_{\alpha(\delta , f^{\delta})}^{\delta}$ satisfies 
the first order optimality condition (\ref{optimality_1}), 
we then have,
\bea
D_{J}(\varphi^{\dagger} , \varphi_{\alpha(\delta , f^{\delta})}^{\delta})
& \leq & (\overline{\tau} + 1)\Psi(\delta) + 
\frac{1}{\alpha(\delta , f^{\delta})} \langle \cT^{\ast}(f^{\delta} - 
\cT\varphi_{\alpha(\delta , f^{\delta})}^{\delta}) , 
\varphi_{\alpha(\delta , f^{\delta})}^{\delta} - \varphi^{\dagger} \rangle .
\nonumber
\eea
Now apply the Cauchy-Schwarz inequality and 
take into account the estimation (\ref{consequence_MDP}), 
$\norm{\cT\varphi_{\alpha(\delta , f^{\delta})}^{\delta} - f^{\delta}}_{\cL^{2}(\cZ)} \leq \overline{\tau}\delta$
for the choice of regularization parameter 
$\alpha(\delta , f^{\delta}) \in \overline{S},$
to arrive at
\bea
D_{J}(\varphi^{\dagger} , \varphi_{\alpha(\delta , f^{\delta})}^{\delta})
& \leq & (\overline{\tau} + 1)\Psi(\delta) +
\frac{1}{\alpha(\delta , f^{\delta})} \norm{\cT\varphi_{\alpha(\delta , f^{\delta})}^{\delta} - f^{\delta}}_{\cL^{2}(\cZ)}
\norm{\cT\varphi_{\alpha(\delta , f^{\delta})}^{\delta} - \cT\varphi^{\dagger}}_{\cL^{2}(\cZ)}
\nonumber\\
& \leq & (\overline{\tau} + 1)\Psi(\delta) + 
\frac{1}{\alpha(\delta , f^{\delta})}\overline{\tau}\delta 
\norm{\cT\varphi_{\alpha(\delta , f^{\delta})}^{\delta} - \cT\varphi^{\dagger}}_{\cL^{2}(\cZ)} 
\nonumber\\
& \leq & (\overline{\tau} + 1)\Psi(\delta) + 
\frac{1}{\alpha(\delta , f^{\delta})}\overline{\tau}(\overline{\tau} + 1)\delta^2
\nonumber\\
& \leq & (\overline{\tau} + 1) \Psi(\delta) + 
\frac{4}{\sigma} \frac{\overline{\tau}(\overline{\tau} + 1)}{(\underline{\tau} - 1)}\Psi(\delta).
\nonumber
\eea
Here, again, the lower bound for the regularization parameter $\alpha(\delta , f^{\delta})$
given in (\ref{index_func_lower_bound1}) has controlled the trade-off
between $\delta^2$ and $\alpha.$ Hence,
this yields the stable upper bound (\ref{theorem_upper_bound_bregman}).
\end{proof}

Upper bounds obtained in the theorems \ref{theorem_Bregman_bound_usual} and \ref{theorem_norm_convergence}
provide upper bound for the symmetric Bregman distance defined in (\ref{symmetrical_bregman}).
The finalizing result of this section can be compared to 
\textbf{\cite[Proof of Theorem 4.4]{Grasmair13}}.

\begin{cor}
\label{corollary_symmetric_bregman_bound}
From the theorems \ref{theorem_Bregman_bound_usual} and \ref{theorem_norm_convergence},
and by the definition given in (\ref{symmetrical_bregman}),
it is concluded that
\bea
D_{J}^{\mathrm{sym}}(\varphi_{\alpha(\delta , f^{\delta})}^{\delta} , \varphi^{\dagger})
= \cO(\Psi(\delta)), \mbox{ as } \delta \rightarrow 0.
\eea
\end{cor}

%================
% Conclusion
%================

\section{Conclusion and Future Prospects}
\label{conclusion}

The goal of this work has been providing a general analysis
for the verification of the generalized variational source
condition given by (\ref{variational_ineq}). Without specification
of the rule for the choice of the regularization parameter, the results
above would not have been obtained. Certainly, further necessary
tool is a stable lower bound for $\alpha(\delta , f^{\delta})$
stated by Theorem \ref{new_lower_bound_regpar}. The condition given in
(\ref{VSC_verification_cond}) has been mentioned in 
\textbf{\cite[Theorem 4.4]{Grasmair13}} but only for the quadratic
Tikhonov functional. 

Further generalization of this work would be possible by considering
the following form of the Tikhonov functional,
\bea
\label{conc_objective_func}
F_{\alpha}(\varphi , f^{\delta}) := 
\frac{1}{q}\norm{\cT\varphi - f^{\delta}}_{\cH}^q + \alpha J(\varphi).
\eea
The order of norm $q$ will change the rates of the error estimation.

Interpretation of this work will be introduced with considering different
penalty terms $J.$ From the early assumption (\ref{total_error_bregman}),
a lower bound which is a function of corresponding norm, say 
$\Phi(\norm{\varphi_{\alpha}^{\delta} - \varphi^{\dagger}}_{\cV})$ 
per different $J$ will permit one to obtain norm convergence result.
With the involvement of any $J$ in (\ref{cost_functional}),
or equivalently in (\ref{conc_objective_func}), defining regularity
properties for the solution function $\varphi$ will be broadened.
To be more specific, different norm convergence and convergence rates
results will also follow from which function space $\cV$ is considered.

\newpage

%===========
% APPENDIX
%===========

\begin{appendices}
\chapter{\textbf{APPENDIX}}
%===============================================================
\section{A Symmetric Form of the Variational Source Condition}
\label{symmetric_VSC}

Here, we question whether it is plausible to state
different form of VSC which is rather associated with
the reverse Bregman distance.

\begin{lemma}
Denote by $\varphi_{\alpha(\delta , f^{\delta})}^{\delta}$
the regularized solution for the problem (\ref{problem0})
where the regularization parameter 
$\alpha(\delta , f^{\delta}) \in \overline{S} \cap \underline{S}$
is chosen {\em a posteriori}. Let $\varphi^{\dagger}$ be the
$J-$minimizing solution introduced in (\ref{J-min_est}).
Then, by this choice of regularization parameter,
the following $J$-difference 
\bea
J(\varphi_{\alpha(\delta , f^{\delta})}^{\delta}) - J(\varphi^{\dagger}) = \cO(\delta^2),
\eea
holds true.
\end{lemma}

\begin{proof}
The regularized solution $\varphi_{\alpha}^{\delta} \in \cD(F_{\alpha})$ 
surely implies 
\bea
\label{J-diff2}
\alpha \left( J(\varphi_{\alpha}^{\delta}) - J(\varphi^{\dagger}) \right) \leq 
\frac{1}{2} \left( \norm{\cT\varphi^{\dagger} - f^{\delta}}_{\cL^{2}(\cZ)}^2 - \norm{\cT\varphi_{\alpha}^{\delta} - f^{\delta}}_{\cL^{2}(\cZ)}^2 \right).
\eea
The right hand side of this inequality can recalculated in the following way,
\bea
\norm{\cT\varphi^{\dagger} - f^{\delta}}_{\cL^{2}(\cZ)}^2 - \norm{\cT\varphi_{\alpha}^{\delta} - f^{\delta}}_{\cL^{2}(\cZ)}^2 
= \langle \cT\varphi^{\dagger} - f^{\delta} , \cT\varphi^{\dagger} - f^{\delta} \rangle - 
\langle \cT\varphi_{\alpha}^{\delta} - f^{\delta} , \cT\varphi_{\alpha}^{\delta} - f^{\delta} \rangle
\nonumber\\
= \langle \cT\varphi^{\dagger} - f^{\delta} , \cT\varphi^{\dagger} - f^{\delta} \rangle - 
\langle \cT\varphi_{\alpha}^{\delta} - f^{\delta} , \cT\varphi_{\alpha}^{\delta} - \cT\varphi^{\dagger} \rangle
- \langle \cT\varphi_{\alpha}^{\delta} - f^{\delta} , \cT\varphi^{\dagger} - f^{\delta} \rangle
\nonumber\\
= \langle \cT\varphi^{\dagger} - \cT\varphi_{\alpha}^{\delta} , \cT\varphi^{\dagger} - f^{\delta} \rangle 
- \langle \cT\varphi_{\alpha}^{\delta} - f^{\delta} , \cT\varphi_{\alpha}^{\delta} - \cT\varphi^{\dagger} \rangle
\eea
We now apply Cauchy-Schwarz ineuqality,
\bea
\norm{\cT\varphi^{\dagger} - f^{\delta}}_{\cL^{2}(\cZ)}^2 - \norm{\cT\varphi_{\alpha}^{\delta} - f^{\delta}}_{\cL^{2}(\cZ)}^2 
\leq \delta \Vert\cT\varphi^{\dagger} - \cT\varphi_{\alpha}^{\delta}\Vert_{\cL^{2}(\cZ)} + 
\Vert \cT\varphi_{\alpha}^{\delta} - f^{\delta} \Vert_{\cL^{2}(\cZ)} \Vert \cT\varphi_{\alpha}^{\delta} - \cT\varphi^{\dagger} \Vert_{\cL^{2}(\cZ)}.
\nonumber
\eea
By the choice of regularization parameter 
$\alpha(\delta , f^{\delta}) \in \overline{S} \cap \underline{S},$
we obtain,
\bea
\norm{\cT\varphi^{\dagger} - f^{\delta}}_{\cL^{2}(\cZ)}^2 - \norm{\cT\varphi_{\alpha(\delta , f^{\delta})}^{\delta} - f^{\delta}}_{\cL^{2}(\cZ)}^2 
& \leq & \delta \Vert\cT\varphi^{\dagger} - \cT\varphi_{\alpha(\delta , f^{\delta})}^{\delta}\Vert_{\cL^{2}(\cZ)} +
\overline{\tau}\delta \Vert \cT\varphi_{\alpha(\delta , f^{\delta})}^{\delta} - \cT\varphi^{\dagger} \Vert_{\cL^{2}(\cZ)}
\nonumber\\
& = &(1 + \overline{\tau})\delta \Vert \cT\varphi_{\alpha(\delta , f^{\delta})}^{\delta} - \cT\varphi^{\dagger} \Vert_{\cL^{2}(\cZ)}
\nonumber
\eea
Further step from here can be estimated by taking into account (\ref{consequence_MDP})
\bea
\norm{\cT\varphi^{\dagger} - f^{\delta}}_{\cL^{2}(\cZ)}^2 - \norm{\cT\varphi_{\alpha(\delta , f^{\delta})}^{\delta} - f^{\delta}}_{\cL^{2}(\cZ)}^2 
\leq (1 + \overline{\tau})^2 \delta^2 .
\nonumber
\eea
Thus, from (\ref{J-diff2}), we arrive at
\bea
\label{J-diff0}
J(\varphi_{\alpha(\delta , f^{\delta})}^{\delta}) - J(\varphi^{\dagger}) \leq 
(1 + \overline{\tau})^2 \frac{\delta^2}{\alpha(\delta , f^{\delta})}.
\eea
\end{proof}

We make use of the estimation (\ref{J-diff0}) to compare 
the well-known index function against generalized reverse 
Bregman distance in the following theorem without avoiding
possible singularity as $\alpha \rightarrow 0.$ This comparison
may give birth to a new form of the VSC.

\begin{theorem}
Denote by $\varphi_{\alpha(\delta , f^{\delta})}^{\delta}$
the regularized solution for the problem (\ref{problem0})
where the regularization parameter 
$\alpha(\delta , f^{\delta}) \in \overline{S} \cap \underline{S}$
is chosen {\em a posteriori}. Let $\varphi^{\dagger}$ be the
$J-$minimizing solution introduced in (\ref{J-min_est}).
Then, for some index function 
$\Psi : [0 , \infty) \rightarrow [0 , \infty),$ the following
\bea
\label{Bregman_Index_func}
D_{J}(\varphi^{\dagger} , \varphi_{\alpha(\delta , f^{\delta})}^{\delta}) \leq \Psi(t),
\eea
holds true.
\end{theorem}

\begin{proof}
Let us begin the proof with assuming the opposite, {\em i.e.} let us assume that
\bea
\label{index_func_contradiction}
D_{J}(\varphi^{\dagger} , \varphi_{\alpha(\delta , f^{\delta})}^{\delta}) > \Psi(t),
\eea
holds for any $t \in [0 , \infty).$ Then, 
by the definition of Bregman distance in (\ref{bregman_divergence_intro}),
for $p \in \partial J(\varphi_{\alpha(\delta , f^{\delta})}^{\delta})$
(\ref{index_func_contradiction}) reads
\bea
J(\varphi^{\dagger}) - J(\varphi_{\alpha(\delta , f^{\delta})}^{\delta})
- \langle p , \varphi^{\dagger} - \varphi_{\alpha(\delta , f^{\delta})}^{\delta} \rangle > \Psi(t)
\nonumber
\eea
Let us include the $J-$difference (\ref{J-diff0}) and rewrite the inner product
\bea
\langle p , \varphi_{\alpha(\delta , f^{\delta})}^{\delta} - \varphi^{\dagger} \rangle - 
(1 + \overline{\tau})^2 \frac{\delta^2}{\alpha(\delta , f^{\delta})} > \Psi(t).
\nonumber
\eea
The regularized solution $\varphi_{\alpha(\delta , f^{\delta})}^{\delta}$
must satisfy the first order optimality condition (\ref{optimality_1}). Thus,
\bea
\frac{1}{\alpha(\delta , f^{\delta})} \langle \cT^{\ast}(f^{\delta} - \cT\varphi_{\alpha(\delta , f^{\delta})}^{\delta}) , 
\varphi_{\alpha(\delta , f^{\delta})}^{\delta} - \varphi^{\dagger} \rangle  - (1 + \overline{\tau})^2 \frac{\delta^2}{\alpha(\delta , f^{\delta})} > \Psi(t).
\nonumber
\eea
By the Cauchy-Schwarz inequality, and by the fact 
that $\alpha(\delta , f^{\delta}) \in \overline{S} \cap \underline{S}$
provides the immediate estimation in (\ref{consequence_MDP}), 
we eventually arrive at the following estimation
\bea
\Psi(t) & < &
\frac{1}{\alpha(\delta , f^{\delta})}\overline{\tau}\delta \norm{\cT\varphi_{\alpha(\delta , f^{\delta})}^{\delta} - \cT\varphi^{\dagger}}_{\cL^{2}(\cZ)}
- (1 + \overline{\tau})^2 \frac{\delta^2}{\alpha(\delta , f^{\delta})} 
\nonumber\\
& < & \frac{1}{\alpha(\delta , f^{\delta})} \overline{\tau} (1 + \overline{\tau})\delta^2 - (1 + \overline{\tau})^2 \frac{\delta^2}{\alpha(\delta , f^{\delta})}
\eea
which is a direct contradiction to the positive definiteness of the index function $\Psi.$
Hence, the assumption in (\ref{index_func_contradiction}) is wrong.
\end{proof}

\end{appendices}

%============================================================================

\newpage

%===========================================================================
% REferences
%===========================================================================
\bigskip
\section*{References}

%\bibliography{../../nfg-ebib/db/ebib}
%\bibliographystyle{annotate}

 \bibliographystyle{alpha}
 % \bibliography{alpha}

\end{document}